\documentclass{amsart}

\usepackage{amssymb}
\usepackage{tikz-cd}
\usepackage{hyperref}

\newtheorem{theorem}{Theorem}[section]
\newtheorem{prop}[theorem]{Proposition}
\newtheorem{lemma}[theorem]{Lemma}
\newtheorem{cor}[theorem]{Corollary}
\newtheorem{assume}[theorem]{Assumption}

\theoremstyle{definition}
\newtheorem{definition}[theorem]{Definition}

\theoremstyle{remark}
\newtheorem{remark}[theorem]{Remark}

\numberwithin{equation}{section}

\begin{document}

\title{Compactness of harmonic maps of surfaces with regular nodes}

\author{Woongbae Park}
\address{Department of Mathematics, Michigan State University, East Lansing, Michigan, 48824}
\email{parkwoo6@msu.edu}

\begin{abstract}
In this paper, we formulate and prove a general compactness theorem for harmonic maps using Deligne-Mumford moduli space and families of curves.
The main theorem shows that given a sequence of harmonic maps over a sequence of complex curves, there is a family of curves and a subsequence such that both the domains and the maps converge off the set of ``non-regular'' nodes.
This provides a sufficient condition for a neck having zero energy and zero length.
As a corollary, the following known fact can be proved: If all domains are diffeomorphic to $S^2$, both energy identity and zero distance bubbling hold.
\end{abstract}

\maketitle

\section{Introduction}
\label{section1}

Given a sequence of harmonic maps with uniformly bounded energy, Uhlenbeck Compactness Theorem discovered ``bubble phenomenon'' where bubble occurs at points of energy concentration.
Later, Parker \cite{P} showed bubble tree extension to describe how a given sequence of harmonic maps converges over bubbles.
For fixed domain, this result is regarded as a full answer of the problem because
\begin{enumerate}
\item
it shows energy identity and zero distance bubbling, and
\item
it specifies where each bubble locates, including bubbles on the bubbles.
\end{enumerate}

However, when domain is varying, we do not have above properties.
The difficulty is when complex structure (or metric) degenerates and there were several studies in this direction.
For example, Chen-Tian \cite{CT} focused on energy minimizing harmonic maps and showed compactness result, together with connecting geodesics.
Zhu \cite{Z} studied the conditions that energy identity and zero distance bubbling do not hold and Chen-Li-Wang \cite{CLW} showed length formula of the neck, but their conditions contain non-geometric quantity.
There are other ways to consider change of complex structure, for example, looking at Teichm{\"u}ller space.
This space restricts to the constant curvature metric and hence seems less appropriate to describe bubbles and necks than Deligne-Mumford moduli space.

In this paper we first define convergence of maps (Definition \ref{conv fam}) in terms of family of complex curves using Deligne-Mumford moduli space.
For each energy concentration points, we put additional marked points to build bigger family in which convergence becomes better in terms of {\it residual energy} (Definition \ref{RE}).
Details of this procedure will be explained in Section \ref{section4}.
Note that energy concentration may occur at regular point or nodal point, where the latter corresponds to a degeneration of complex structure.
Since not always energy identity holds in the case of degenerating complex structure, we need more refined notion of {\it regular node} (Definition \ref{reg node def}).
Now our main theorem can be stated in terms of regular node.

\begin{theorem}\label{main theorem}
Suppose $f_{k} : C_{k} \rightarrow X$ be a sequence of harmonic maps with uniformly bounded energy defined on smooth $(g,n)$ curves.

Then there is a subsequence $n_{k}$ and a way of marking points $P_{k}$ on $C_{k}$ such that corresponding sequence $f_{n_{k}} : C'_{n_{k}} = (C_{n_{k}},P_{n_{k}}) \rightarrow X$ converges to some $f_{0} : C_{0} \rightarrow X$ off the singular set $S$ (possibly empty) in $C^{1}$ where all points in $S$ are non-regular nodal points.
Furthermore, $f_{0}$ is harmonic on closure of each component of $C_{0} \setminus S$ separately.
\end{theorem}

With additional assumption, we can make the singular set $S$ empty.

\begin{cor} \label{main2}
With the same assumption of Theorem \ref{main theorem}, also assume that $C_{k} \rightarrow C_{0}$ in a family $\mathcal{C}$ and all nodes in $C_{0}$ are regular.

Then the singular set $S$ in the convergence in Theorem \ref{main theorem} is empty.
Furthermore, the energy identity holds and the image of $f_{0}$ is connected.
\end{cor}

Now Parker's theorem is a corollary of the main theorem.

\begin{cor}\label{fixed domain}
Let $\Sigma$ be a smooth Riemann surface with genus $g$ and suppose $f_{k} : \Sigma \rightarrow X$ be a sequence of harmonic maps with uniformly bounded energy.

Then Corollary \ref{main2} can be applied with $C_k = C_0 = \Sigma$.
\end{cor}

Another corollary is when $g=0$, that is, all domains are $n$-marked sphere $S^2$.

\begin{cor}\label{sphere}
Let $C_{k}$ be 2-spheres with $n$-marked points and suppose $f_{k} : C_{k} \rightarrow X$ be a sequence of harmonic maps with uniformly bounded energy.

Then Corollary \ref{main2} can be applied.
\end{cor}

Proofs will be given in Section \ref{section7}.

The remaining parts of this paper is organized as follows.
Section \ref{section2} deals with basic properties of harmonic maps and Deligne-Mumford moduli space.
In Section \ref{section3}, we develop necessary convergence terminology.
Section \ref{section4} focuses on neck analysis and Section \ref{section5} introduces regular nodes.
In Section \ref{section6} we explain the procedure of building bigger family by putting appropriate marked points.
Finally, the last section contains proof of the main theorem.

The author thanks to his advisor, Thomas H. Parker, for valuable advices, comments and inspirations about this paper.

\section{Background} 
\label{section2}

\subsection{Harmonic Maps}

Let $(\Sigma, g)$ and $(X,h)$ be compact Riemannian manifolds with Riemannian metrics $g$ and $h$ with $\dim(\Sigma) = 2$.
We use the same letter $g$ to denote Riemannian metric of $\Sigma$ and genus of $\Sigma$ if there is no confusion.
A map $f : (\Sigma,g) \rightarrow (X,h)$ is harmonic if it is a critical point of the energy functional
\begin{equation}\label{energy func}
E(f) = \mathcal{L}(f)= \frac{1}{2}\int_{\Sigma}\lvert df \rvert ^{2} dvol_{g}.
\end{equation}
Using $f_{k}$ and $g$, we can define corresponding energy density measures $e(f_{k})$ on $\Sigma$ by
\begin{equation}\label{e(f)}
e(f_{k}) = \frac{1}{2}\lvert df_{k}\rvert^{2} dvol_{g}.
\end{equation}

We summarize some important lemmas regarding harmonic maps.
Here we follow \cite{SU} and \cite{P}.
For more results, see \cite{EL} or \cite{EL2}.

\begin{theorem}
Suppose $f: (\Sigma,g) \rightarrow (X,h)$ be harmonic.
Then we have the followings:
\begin{enumerate}
\item
($\varepsilon$-regularity)
There is a constant $\varepsilon_{0}>0$ depending only on second fundamental form of the embedding $X \hookrightarrow \mathbb{R}^{N}$ such that if $f$ is a harmonic map on a disk $D$ and if $E_{D}(f) = \frac{1}{2}\int_{D}\lvert df \rvert ^{2} dvol_{g} < \varepsilon_{0}$, then for any $D' \subset \subset D$,
\begin{equation}
\lVert df \rVert_{W^{1,p}(D')} \leq C \lVert df \rVert_{L^{2}(D)},
\end{equation}
where $1<p<\infty$ and $C$ is a constant which only depends on $p$, $D'$ and the geometry of $X$.
\item
(Energy-gap)
There is a constant $\varepsilon'_{0}>0$ depending only on $(X,h)$ such that if $f$ is a smooth harmonic map on a compact domain $\Sigma$ satisfying $E(f) = \frac{1}{2}\int_{\Sigma}\lvert df\rvert ^{2} dvol_{g} < \varepsilon'_{0}$, then $f$ is constant.
\item
(Removable Singularity)
If $f : D \setminus \{0\} \rightarrow X$ is a $C^{1}$ harmonic map with $E(f) < \varepsilon_{0}$ on a punctured disk $D \setminus \{0\}$, then $f$ can be extended to $D$ in $C^{1}$.
\item
($C^{1}$-convergence)
There is a constant $\varepsilon_{0}>0$ such that if $\{f_{k}\}$ is a family of harmonic maps on $D$ and satisfying $E_{D}(f_{k}) < \varepsilon_{0}$ for all $k$, then there is a subsequence $f_{k}$ that converges to $f$ in $C^{1}$.
\end{enumerate}
\end{theorem}

Using these, Uhlenbeck proved the following compactness theorem.

\begin{theorem}\label{cpt}
(Uhlenbeck Compactness Theorem, \cite[Theorem 4.4]{SU} or \cite[Lemma 1.2]{P})
Suppose $\{f_{k}\}$ be a sequence of harmonic maps with uniformly bounded energy.
Then there are at most finite number of points $\{p_{1}, \ldots, p_{l}\}$, called bubble points, subsequence of $\{f_{k}\}$ and limit map $f_{\infty} : (\Sigma,g) \rightarrow (X,h)$ such that $f_{k} \rightarrow f_{\infty}$ in $C^{1}$ for any compact set away from $\{p_{1}, \ldots, p_{l}\}$, and
\begin{equation}\label{measure conv1}
e(f_{k}) \rightarrow e(f_{\infty}) + \sum_{i=1}^{l} m_{i} \delta_{p_{i}}
\end{equation}
as measures where $m_{i} \geq \varepsilon'_{0}$.
\end{theorem}

Parker \cite{P} used iterated renormalizations to construct a so-called ``bubble tree" to analyze the energy completely, i.e., all the energy comes from either the limit map or bubbles.

\begin{theorem}\label{bubble tree}
(Bubble Tree Convergence, \cite[Theorem 2.2]{P}).
Under the same assumption of Theorem \ref{cpt}, there is a subsequence $\{f_{n}\}$ and a bubble tower domain $T = \Sigma \cup \bigcup S_{I}$ so that the renormalized maps
\begin{equation*}
\{f_{n,I}\} : T \rightarrow X
\end{equation*}
converges in $W^{1,2} \cap C^{0}$ to a smooth harmonic bubble tree map $\{f_{I}\} : T \rightarrow X$.
Moreover,
\begin{enumerate}
\item
(No energy loss)
$E(f_{n})$ converges to $\sum E(f_{I})$, and
\item
(Zero distance bubbling)
At each bubble point $x_{J}$ (at an level in the tree), the images of the base map $f_{I}$ and the bubble map $f_{J}$ meet at $f_{I}(x_{J}) = f_{J}(p^{-})$.
\end{enumerate}
\end{theorem}

\subsection{Families of Curves}

In this subsection we define a notion of convergence of complex curves $C_{k}$.
As in algebraic geometry, it is useful and important to consider not just single curves, but instead {\it families} of curves.
There are standard definitions of such families used by algebraic geometers.
Many details can be found at \cite{ACG}.
We will use an equivalent definition that is more in the spirit of differential geometry.
For details, see \cite{RS}.

\begin{definition}
A $(g,n)$ curve is defined to be a connected $n$-marked nodal curve of genus $g$, that is, a complex curve $C$ of genus $g$ with at most nodal singularity together with a sequence $\{x_{1}, \ldots, x_{n}\}$ of distinct points of $C$.
A $(g,n)$ curve is said to be stable if $2g-2+n>0$.
\end{definition}

\begin{definition}\label{nodal family}
(\cite{RS} 4.2,4.7)
An $n$-marked nodal family is a surjective proper holomorphic map $\pi : \mathcal{C} \rightarrow B$ between connected complex manifolds with disjoint submanifolds, together with $\mathcal{N},S_{1}, \ldots, S_{n}$ such that
\begin{enumerate}
\item
$\dim_{\mathbb{C}}(\mathcal{C}) = \dim_{\mathbb{C}}(B) + 1$,
\item
$\pi|_{S_{i}}$ maps $S_{i}$ diffeomorphically onto $B$,
\item
$\mathcal{N} = \{p \in \mathcal{C} : d\pi(p) \textrm{ not surjective}\}$, and
\item
\sloppy Each critical point $p \in \mathcal{N}$ has local holomorphic coordinate chart $(x,y,z_{2},\ldots , z_{n})$, called a nodal chart, such that
\begin{equation}
\pi(x,y,z_{2},\ldots,z_{n}) = (xy,z_{2},\ldots,z_{n})
\end{equation}
is a local holomorphic coordinate chart in a neighborhood of $\pi(p)$.
\end{enumerate}
We call $\mathcal{N}$ the nodal set, and $S_{i}$ the marked sections.
\end{definition}

\sloppy By the holomorphic Implicit Function Theorem, $p \in \mathcal{C} \setminus \mathcal{N}$ has local holomorphic coordinate chart $(x,z_{1},z_{2},\ldots,z_{n})$, called a regular chart, such that $\pi(x,z_{1},z_{2},\ldots,z_{n}) = (z_{1},z_{2},\ldots,z_{n})$ is a local holomorphic coordinate chart in a neighborhood of $\pi(p)$.
Note that $\mathcal{N}$ intersects each fiber $C_{b} := \pi^{-1}(b)$ in a finite set.
For each regular value $b \in B$ of $\pi$ the fiber $C_{b}$ is a compact Riemann surface.
For each critical value $b \in B$ of $\pi$ the fiber $C_{b}$ is a nodal curve.

Together with an isomorphism between a curve $C$ and a fiber, one can consider a deformation of $C$ as follows.

\begin{definition}\label{deform}
(\cite{ACG} 11.2.1, 11.4.2)
Let $C$ be a (marked) complex curve (or a nodal curve).
A deformation of $C$ is a nodal family $\pi : \mathcal{C} \rightarrow (B,b_{0})$ plus a given isomorphism between $C$ and the central fiber $\pi^{-1}(b_{0})$ that maps marked points to marked sections.
\end{definition}

We understand the concept of deformation in terms of a germ.
Hence a restriction of a deformation containing the central fiber is regarded as equivalent as the original one.

Now we describe the Deligne-Mumford moduli space.
The Deligne-Mumford moduli space of stable $n$-marked nodal curves of genus $g$ is defined by
\begin{equation}\label{DM}
\overline{\mathcal{M}}_{g,n} := \{ \textrm{ isomorphism classes } [C] \textrm{ of } (g,n) \textrm{ curves }\}.
\end{equation}

The structure of \eqref{DM} is well-known.
Here we point out properties of \eqref{DM} that will be used in this paper.
\begin{itemize}
\item
Let $\mathcal{M}_{g,n}$ be a moduli space of stable $n$-marked smooth curves of genus $g$.
Then $\overline{\mathcal{M}}_{g,n}$ is its compactification.
\item
$\overline{\mathcal{M}}_{g,n}$ is a complex projective variety, and has orbifold structure.
\item
There is a projection map, called forgetful map, given by
\begin{equation}
\Phi : \overline{\mathcal{M}}_{g,n+1} \rightarrow \overline{\mathcal{M}}_{g,n}
\end{equation}
which forgets the last marked point and collapses unstable component to a point.
\end{itemize}

There are several different ways to construct $\overline{\mathcal{M}}_{g,n}$.
Here we see a neighborhood of $[C_{0}]$ in $\overline{\mathcal{M}}_{g,n}$ as a quotient of Kuranishi family of $C_{0}$ by a finite group $Aut(C_{0})$ of automorphisms of central fiber $C_{0}$.
If $[C_{k}] \rightarrow [C_{0}]$ in $\overline{\mathcal{M}}_{g,n}$ with $Aut(C_{0})$ being not trivial, then the quotient has a singularity at $C_{0}$.
Instead, we use Kuranishi family $\pi : \mathcal{C} \rightarrow (B,0)$ of one of representatives $C_{0}$ of $[C_{0}]$ without quotient.
Then the total space $\mathcal{C}$ is a smooth manifold, but $Aut(C_{0})$ acts as transformation of fibers, which destroys uniqueness of embedding of $C_{k}$ into family.

\begin{definition} \label{conv of curve}
For sequence of $(g,n)$ curves $C_{k}$, we say $C_{k} \rightarrow C_{0}$ in a family $\mathcal{C}$ if $\pi : \mathcal{C} \rightarrow (B,0)$ is a deformation of $C_{0}$ with isomorphism $\varphi : C_{0} \rightarrow \pi^{-1}(0)$ and isomorphisms $\varphi_{k} : C_{k} \rightarrow \pi^{-1}(b_{k})$ with $b_{k} \rightarrow 0$.
Since $\varphi$ and $\varphi_{k}$ are isomorphisms, we identify $C_{k}$ with $\pi^{-1}(b_{k})$ and $C_{0}$ with $\pi^{-1}(0)$.
\end{definition}

\begin{remark}
There is no uniqueness statement for the choice of the sequence $b_{k}$ and the isomorphisms $\varphi$ and $\varphi_{k}$.
However, all of the convergence statements below hold for any choice.
\end{remark}

\section{Convergence of Maps} 
\label{section3}

Now we care about a sequence of $C^{l}$ (of $W^{l,p}$) maps $f_{k} : C_{k} \rightarrow X$ with $C_{k} \rightarrow C_{0}$ in a family $\pi : \mathcal{C} \rightarrow (B,0)$.
In a regular chart of the family, we can give a local coordinate of $C_{k}$ by $(x,b_{k})$ and $C_{0}$ by $(x,0)$.
Also, assume that $E(f_{k}) \leq E_{0} < +\infty$ for all $k$.

\begin{definition}\label{znp}
Let $C_{k} \rightarrow C_{0}$ in a family $\mathcal{C}$ and $p$ be a node of $C_{0}$.
We say a sequence of maps $f_{k} : C_{k} \rightarrow X$ satisfies the zero neck property at $p$ if the following holds:

For any $\varepsilon>0$, there is $\delta_{1}>0$ such that for any $0<\delta<\delta_{1}$ and for all $k$ sufficiently large,
\begin{align*}
E(f_{k}, B(p,\delta) \cap C_{k}) &\leq \varepsilon\\
\mathrm{diam}(f_{k}( B(p,\delta) \cap C_{k})) &\leq \varepsilon
\end{align*}
where $B(p,\delta)$ is a ball of radius $\delta$ centered at $p$ in the family $\mathcal{C}$.
\end{definition}

\begin{definition} \label{conv fam}
We say $f_{k}$ converges to $f_{0} : C_{0} \rightarrow X$ off the set $S \subset C_{0}$ in $C^{l}$ (or $W^{l,p}$) if
\begin{enumerate}
\item
$C_{k} \rightarrow C_{0}$ in a family $\mathcal{C}$,
\item
for any node $p \notin S$, $f_k$ satisfies the zero neck property at $p$.
\item
in every regular chart away from $S$, the projected map
\begin{equation} \label{tilde f}
\tilde{f}_{k}(x) := f_{k} (x,b_{k})
\end{equation}
converges to $\tilde{f}_{0}(x) := f_{0}(x,0)$ in $C^{l}$ (or $W^{l,p}$).
\end{enumerate}

The set $S$ is called singular set.
\end{definition}

\begin{lemma}
(Energy Identity and Connected Image)
Suppose $f_{k} : C_{k} \rightarrow X$ converges to $f_{0} : C_{0} \rightarrow X$ in $C^{1}$ with empty singular set.
Then 
\begin{equation}
\lim_{k \rightarrow \infty}E(f_{k}) = E(f_{0}).
\end{equation}
Furthermore, if each $C_{k}$ is connected, then the image of $f_{0}$ is connected.
\end{lemma}

\begin{proof}
Pick any regular point $p \in C_{0}$ and consider a regular chart $\pi : U \times B \rightarrow B$ with $p=(0,0)$.
The projection map $\pi_{k} : U \times \{b_{k}\} \rightarrow U$ given by $\pi_{k}(x,b_{k}) = x$ is holomorphic, so the energy of $f_{k}$ over $U \times \{b_{k}\}$ is the same as the energy of $\tilde{f}_{k}$ over $U$.
Since $f_{k}$ converges to $f_{0}$ in $C^{1}$, $\lim_{k \rightarrow \infty}E(f_{k},U \times \{b_{k}\}) = E(f_{0},U)$.
Also, because of the zero neck property, we have
\begin{equation*}
\lim_{\delta \rightarrow 0} \lim_{k \rightarrow \infty} E(f_{k}, B(q,\delta) \cap C_{k}) = 0
\end{equation*}
for any node $q$.
Cover $C_{0}$ by finitely many regular charts $\{U_{i} \times B\}$ and we have
\begin{equation*}
\lim_{k \rightarrow \infty}E(f_{k}) = E(f_{0}).
\end{equation*}
Connectedness comes from the other condition of the zero neck property.
\end{proof}

Using this new definition of convergence, Uhlenbeck Compactness Theorem \ref{cpt} can be rewritten as follows:
Note that in regular chart $\tilde{f}_{k}(x) := f_{k}(x,b_{k})$ as in \eqref{tilde f}.

\begin{lemma}\label{loc conv}
(Local convergence)
Suppose a sequence of smooth $(g,n)$ curves $C_{k}$ converges to $C_{0}$ in a family $\mathcal{C}$ and $N$ be a nodal set in $C_{0}$.
Let $f_{k} : C_{k} \rightarrow X$ be a sequence of harmonic maps with uniformly bounded energy.

Then there is a subsequence $n_{k}$, a finite set $Q = \{p_{1}, \ldots, p_{l}\} \subset C_{0} \setminus N$, and a harmonic map $f_{0} : C_{0} \rightarrow X$ such that $f_{n_{k}}$ converges to $f_{0}$ off the set $S = Q \cup N$ in $C^{1}$.
Also, near $p \in Q \subset C_{0}$, we have
\begin{equation}\label{meas conv}
e(\tilde{f}_{n_{k}}) \rightarrow e(f_{0}) + m_{p}\delta_{p},
\end{equation}
where $\delta_{p}$ is a Dirac-delta measure at $p$ and $m_{p} \geq \varepsilon'_{0}$.
\end{lemma}

\begin{proof}
Cover $C_{0} \setminus N$ by finitely many regular chart $\{U_{i} \times B\}$.
Since the regular chart is holomorphic coordinate chart, $\tilde{f}_{k}(x)$ is also harmonic maps over $U_{i}$ for each $i$.
By Uhlenbeck compactness theorem \ref{cpt}, there is a subsequence $n_{k,1}$ such that $\tilde{f}_{n_{k,1}}$ converges on $U_{1} \setminus Q_{1}$ where $Q_{1}$ is a finite set.
Applying Theorem \ref{cpt} again to find subsequence $n_{k,2}$ of $n_{k,1}$ such that $\tilde{f}_{n_{k,2}}$ converges on $U_{2} \setminus Q_{2}$ where $Q_{2}$ is a finite set.
Repeat this process to obtain $n_{k,i}$ for $U_{i}$, and choose diagonal subsequence $n_{k} = n_{k,k}$.
Then $\tilde{f}_{n_{k}}$ converges on $U_{i} \setminus Q_{i}$ for all $i$, where $Q_{i}$ is a finite set.
Note that at each point $p \in Q_{i}$, by \eqref{measure conv1}, the energy concentration at $p$ is at least $\varepsilon'_{0}$.
Since the total energy is finite, $Q = \cup Q_{i}$ is at most finite.
The measure convergence \eqref{meas conv} comes from \eqref{measure conv1}.
\end{proof}

The above lemma says that energy loss may occur only at points in $S$.
\begin{definition}\label{bubble pt}
In Lemma \ref{loc conv}, we say $p \in Q$ a smooth bubble point and $p \in N$ a nodal bubble point.
\end{definition}

The above lemma only says about smooth bubble points and nothing about nodes.
So the question is: does $f_k$ satisfies the zero neck property at each node?
We will see one sufficient condition that this is true in Section \ref{section5}.

\section{Neck analysis}
\label{section4}

In this section we will deal with sequence of harmonic maps over the neck region.
Throughout this section, we will use both nodal chart or cylindrical chart that are described below.
Many of the arguments in here can be found in \cite{P}, \cite{Z}, or \cite{LW2}.

First we set up the terminology.
Consider a sequence of harmonic maps $f_{k} : C_{k} \rightarrow X$ defined on family of stable curves $C_{k}$ with $C_{k} \rightarrow C_{0}$ where $C_{0}$ is a limit curve with a node $p$.
For any $\delta>0$, using nodal chart, we can denote $B(p,\delta) \cap C_{k}= \{x,y \in \mathbb{C}^{2} : xy = t_{k}, \lvert x \rvert, \lvert y \rvert \leq \delta\}$ and $t_{k} \rightarrow 0$ as $k \rightarrow \infty$.
Consider polar coordinate $x = (r,\theta)$, and take log by $(r,\theta) \rightarrow (t,\theta)$ where $t = \ln(r/\sqrt{t_{k}})$.
This gives a cylindrical coordinate over the neck which is conformal to the original nodal chart, given by
\begin{equation} \label{cyl coord}
\phi : [-T_{k}^{\delta},T_{k}^{\delta}] \times S^{1} \rightarrow B(p,\delta) \cap C_{k}
\end{equation}
where $\phi(t,\theta) = (x,y) = (\sqrt{t_{k}} e^{t+i \theta}, \sqrt{t_{k}} e^{-t-i \theta})$ and $T_{k}^{\delta} = \ln(\delta/\sqrt{t_{k}})$.
Note that pullback of the metric to the cylinder is conformally equivalent to $dt^{2}+d\theta^{2}$, so we can consider flat metric over the cylinder.
For simplicity, we often denote $f_{k} \circ \phi$ by $f_k$, or simply by $f$.

By \cite[Lemma 3.3]{Z}\label{alpha}, the quantity
\begin{equation*}
\alpha(t) := \frac{1}{2} \int_{\{t\} \times S^{1}} \lvert f_{t} \rvert^{2} - \lvert f_{\theta} \rvert^{2} d\theta
\end{equation*}
is a constant $\alpha$, which is independent of $t$. Hence, the energy can be written as
\begin{equation} \label{energy theta}
E(f) = \frac{1}{2} \iint_{[-T_{k}^{\delta},T_{k}^{\delta}] \times S^{1}} \lvert f_{t} \rvert^{2} + \lvert f_{\theta} \rvert^{2} dt\,d\theta = 2T_{k}^{\delta} \alpha + \int_{-T_{k}^{\delta}}^{T_{k}^{\delta}} \Theta dt
\end{equation}
where
\begin{equation*}
\Theta(t) := \int_{\{t\} \times S^{1}} \lvert f_{\theta} \rvert^{2} d\theta.
\end{equation*}
Moreover, by \cite[Lemma 3.1]{Z}, there exists $\varepsilon_{0}''>0$ such that if $f : [-T_{k}^{\delta},T_{k}^{\delta}] \times S^{1} \rightarrow X$ is a harmonic map and $E(f) \leq \varepsilon_{0}''$, then
\begin{equation*}
\Theta'' \geq \Theta > 0
\end{equation*}
and that for any $-T_{k}^{\delta} \leq T_{1} < T_{2} \leq T_{k}^{\delta}$, we have
\begin{align*}
\int_{T_{1}}^{T_{2}} \Theta dt &\leq 2(\Theta(T_{1}) + \Theta(T_{2})),\\
\int_{T_{1}}^{T_{2}} \sqrt{\Theta} dt &\leq 4(\sqrt{\Theta(T_{1})} + \sqrt{\Theta(T_{2})}).
\end{align*}

From the above, energy is controlled by $\alpha$ and boundary values of $\Theta$.
We need one more quantity.
Let the average length of the neck be given by:
\begin{equation*}
\overline{L} := \frac{1}{2\pi}\iint_{[-T_{k}^{\delta},T_{k}^{\delta}] \times S^{1}} \lvert f_{t} \rvert dtd\theta.
\end{equation*}
Then
\begin{align*}
\overline{L} &\leq \frac{1}{2\pi} \int_{-T_{k}^{\delta}}^{T_{k}^{\delta}} \sqrt{2\pi} \left( \int_{S^{1}}\lvert f_{t} \rvert^{2} d\theta \right)^{1/2} dt \\
&\leq \frac{1}{2\pi} \int_{-T_{k}^{\delta}}^{T_{k}^{\delta}} \sqrt{2\pi} \left[ \left( \int_{S^{1}}\lvert f_{t} \rvert^{2} - \lvert f_{\theta} \rvert^{2} d\theta \right)^{1/2}  + \left( \int_{S^{1}} \lvert f_{\theta} \rvert^{2} d\theta \right)^{1/2} \right] dt \\
&= \frac{2}{\sqrt{2 \pi}} \sqrt{\alpha}T_{k}^{\delta} + \frac{1}{\sqrt{2\pi}} \int_{-T_{k}^{\delta}}^{T_{k}^{\delta}} \sqrt{\Theta}dt.
\end{align*}

Now we state and prove the main proposition of the section.

\begin{prop} \label{gap2}
Suppose $f_{k} : [-T_{k}^{\delta},T_{k}^{\delta}] \times S^{1} \rightarrow X$ is a sequence of harmonic maps and $E(f_{k}) \leq \varepsilon_{0}''$ for all $k$.
For any $\varepsilon>0$, there is $\delta_{1}>0$ such that for any $0<\delta<\delta_{1}$ and for all $k$ sufficiently large,
\begin{align*}
E(f_{k}, [-T_{k}^{\delta},T_{k}^{\delta}] \times S^{1}) &\leq 2T_{k}^{\delta} \alpha + \varepsilon,\\
\mathrm{diam}(f_{k}(  [-T_{k}^{\delta},T_{k}^{\delta}] \times S^{1}) &\leq C \sqrt{\alpha} T_{k}^{\delta} + \varepsilon.
\end{align*}
\end{prop}

\begin{proof}
Consider the nodal chart of the neck $B(p,\delta) \cap C_{k}$ and $C_{k} \rightarrow C_{0}$ in a family $\mathcal{C}$.
Away from the node $p$ and by choosing subsequence, $f_{k} \rightarrow f_{0}$ in $C^{1}$ for some $f_{0} : C_{0} \setminus \{p\} \rightarrow X$.
For simplicity, denote $B_{\delta} = B(p,\delta)$.
We first claim that, for any $\varepsilon>0$, there is $\delta_1>0$ such that for any $0<\delta<\delta_1$ and for all $k$ sufficiently large,
\begin{equation} \label{theta eq3}
\int_{-T_{k}^{\delta}}^{T_{k}^{\delta}} \Theta dt \leq \varepsilon, \qquad \int_{-T_{k}^{\delta}}^{T_{k}^{\delta}} \sqrt{\Theta} dt \leq \varepsilon, \qquad \textrm{ and } \qquad \int_{\{t\} \times S^{1}} \lvert f_{\theta} \rvert d\theta \leq \varepsilon.
\end{equation}

Fix $\varepsilon>0$.
Choose $\delta_{1}>0$ such that for all $\delta < \delta_1$,
\begin{enumerate}
\item
$E(f_0, (B_{\delta} \setminus B_{\delta/2}) \cap C_0) \le \varepsilon/2$, and
\item
the length of $f_{0}(\lvert x \rvert = \delta)$ and $f_{0}(\lvert y \rvert = \delta)$ are less than $\varepsilon/2$.
\end{enumerate}

Fix such $\delta$, then for all $k$ sufficiently large, 
\begin{enumerate}
\item
$E(f_{k},(B_{\delta} \setminus B_{\delta/2}) \cap C_{k}) \le \varepsilon$, and
\item
$\displaystyle \int_{\{\pm T_{k}^{\delta}\} \times S^{1}} \lvert f_{\theta} \rvert d\theta \leq \varepsilon$.
\end{enumerate}

Using cylindrical domain, the first energy bound inequality means that
\begin{equation*}
E(f_{k},B_{\delta} \setminus B_{\delta/2}) = \int_{-T_{k}^{\delta}}^{-T_{k}^{\delta}+\ln2} \int_{S^{1}} \lvert f_{t} \rvert^{2} + \lvert f_{\theta} \rvert^{2} d\theta\, dt + \int_{T_{k}^{\delta} - \ln2}^{T_{k}^{\delta}} \int_{S^{1}}\lvert f_{t} \rvert^{2} + \lvert f_{\theta} \rvert^{2} d\theta\, dt \leq \varepsilon.
\end{equation*}

Now fix $\delta,k$.
Choose $a \in (-T_{k}^{\delta},-T_{k}^{\delta}+\ln2), b \in (T_{k}^{\delta}-\ln2, T_{k}^{\delta})$ such that
\begin{equation*}
\Theta(a) = \frac{1}{\ln2}\int_{-T_{k}^{\delta}}^{-T_{k}^{\delta}+\ln2} \Theta dt, \qquad \Theta(b) = \frac{1}{\ln2}\int_{T_{k}^{\delta}-\ln2}^{T_{k}^{\delta}} \Theta dt.
\end{equation*}
Then,
\begin{align*}
\int_{-T_{k}^{\delta}}^{T_{k}^{\delta}} \Theta dt &= \int_{-T_{k}^{\delta}}^{a} \Theta dt + \int_{b}^{T_{k}^{\delta}} \Theta dt + \int_{a}^{b} \Theta dt\\
&\leq \int_{-T_{k}^{\delta}}^{-T_{k}^{\delta}+\ln2} \Theta dt+ \int_{T_{k}^{\delta}-\ln2}^{T_{k}^{\delta}} \Theta dt + 2 ( \Theta(a) + \Theta(b))\\
&= (1 + \frac{2}{\ln2}) \left( \int_{-T_{k}^{\delta}}^{-T_{k}^{\delta}+\ln2} \Theta dt+ \int_{T_{k}^{\delta}-\ln2}^{T_{k}^{\delta}} \Theta dt \right) \leq (1 + \frac{2}{\ln2}) \varepsilon.
\end{align*}

Second inequality in \eqref{theta eq3} can be obtained in the similar manner.

For the last inequality in \eqref{theta eq3}, using $\varepsilon$-regularity, 
\begin{align*}
\int_{\{t\} \times S^{1}} \lvert f_{\theta} \rvert d\theta &\leq \sqrt{2\pi} \sqrt{\Theta(t)} \leq \sqrt{2 \pi} \sqrt{\Theta(\pm T_{k}^{\delta})} = \sqrt{2 \pi} \left(\int_{\{\pm T_{k}^{\delta}\} \times S^{1}} \lvert f_{\theta} \rvert^{2} d\theta \right)^{1/2}\\
&\leq \sqrt{2 \pi} \sqrt{\sup \lvert df \rvert} \left(\int_{\{\pm T_{k}^{\delta}\} \times S^{1}} \lvert f_{\theta} \rvert d\theta \right)^{1/2}\\
&\leq C \sqrt{\varepsilon''_{0}} \sqrt{\varepsilon}.
\end{align*}

Now $E \leq 2T_{k}^{\delta} \alpha + \varepsilon$ is clear from \eqref{energy theta}.
To see the diameter, choose $\theta_{0}$ such that $\overline{L} = L_{\theta_{0}}$.
Let $(t_{1},\theta_{1})$ and $(t_{2},\theta_{2})$ be such that
\begin{equation*}
\max_{x,y \in [-T_{k}^{\delta},T_{k}^{\delta}] \times S^{1}} (\lvert f(x) - f(y) \rvert) = \lvert f(t_{1},\theta_{1}) - f(t_{2},\theta_{2}) \rvert.
\end{equation*}
Then,
\begin{align*}
\mathrm{diam} f &\leq \lvert f(t_{1},\theta_{1}) - f(t_{1},\theta_{0}) \rvert + \lvert f(t_{1},\theta_{0})-f(t_{2},\theta_{0}) \rvert + \lvert f(t_{2},\theta_{0}) - f(t_{2},\theta_{2}) \rvert\\
&\leq \int_{\{t_{1}\} \times S^{1}} \lvert f_{\theta} \rvert d\theta + L_{\theta_{0}}+ \int_{\{t_{2}\} \times S^{1}} \lvert f_{\theta} \rvert d\theta \leq 2\varepsilon + \overline{L}.
\end{align*}
This completes the proof.
\end{proof}

\section{Continuity at Regular nodes}
\label{section5}

In general, $\alpha$ in the previous section is not small enough to obtain the zero neck property at each node.
However, if the node is {\em regular}, $\alpha$ vanishes and hence the zero neck property holds.
Here the existence of forgetful map is crucial.
We first define the notion of regular node.

\begin{definition}\label{reg node def}
Let $\pi : \mathcal{C} \rightarrow B$ be a family of $(g,n)$ curves and $p$ be a nodal point.
We say $p$ is a regular node if there exists a family of $(g,l)$ curves $\overline{\pi} : \overline{\mathcal{C}} \rightarrow \overline{B}$ with $l<n$ and a forgetful map $\Phi : \mathcal{C} \rightarrow \overline{\mathcal{C}}$ such that $\overline{p} := \Phi(p)$ is a regular point.
\end{definition}

\begin{lemma}\label{gap neck}
(Energy gap in the neck)
Suppose $p$ be a regular node and $E(f_{k}, B(p,\delta_{0}) \cap C_{k}) \leq \varepsilon'_{0}$ for some $\delta_{0}>0$ and for all $k$.
Then $f_{k}$ satisfies the zero neck property at $p$.
\end{lemma}

Lemma \ref{gap neck} provides another energy quantization $\varepsilon''_{0}$ over regular nodes.
Combined the energy quantization $\varepsilon'_{0}$ at smooth bubble points, these quantizations are the key of the definition of residual energy.
Proof of this lemma will be given at the end of the section.

\begin{lemma}\label{reg node}
Suppose $f_{k} : C_{k} \rightarrow X$ is a sequence of harmonic maps, $C_{k} \rightarrow C_{0}$ and $p \in C_{0}$ is a regular node.
Then in the cylindrical coordinate,
\begin{equation*}
\int_{0}^{2\pi} \lvert f_{t} \rvert^{2} d\theta = \int_{0}^{2\pi} \lvert f_{\theta} \rvert^{2} d\theta
\end{equation*}
which means that $\alpha = 0$.
\end{lemma}

\begin{proof}
Let $\Phi : \mathcal{C} \rightarrow \overline{\mathcal{C}}$ be a forgetful map such that $\overline{p} = \Phi(p)$ is a regular point.
Choose a regular chart $(z,\overline{b})$ of $\overline{\mathcal{C}}$ centered at $\overline{p} = (0,0)$ and denote $B_{\delta} \subset \overline{C_0}$ by a ball of radius $\delta$ centered at $\overline{p}$ in this chart.
Then $\Phi^{-1}(B_{\delta} \setminus \{\overline{p}\})$ is biholomorphic to a punctured disk $\mathring{B} \subset C_0$ while $\Phi^{-1}(\overline{p})$ is a union of components with marked points on them.
Without loss of generality, we can assume that $p$ is the puncture of $\mathring{B}$ in which the other component meets.
Pick a nodal chart near $p$ of $C_{k} \cap B(p,\delta)$ given by $(x,y,b)$ with $x \in \mathring{B}$ such that $xy=t_{k}$ for some $t_{k}$ and $|x|,|y| \leq \delta$.
If necessary, by modifying the charts, we can arrange that $\Phi(x,y,b) = (z,\overline{b})$ with $z=x$.

Denote $F_{k} := f_{k} \circ \Phi^{-1} : \overline{C_{k}} \rightarrow X$, which is also harmonic.
By Pohozaev identity,
\begin{equation} \label{Poho}
\int_{0}^{2\pi} \big\lvert F_{\phi} \big\rvert^{2} d\phi = r^{2} \int_{0}^{2\pi} \big\lvert F_{r} \big\rvert^{2} d\phi
\end{equation}
where $(r,\phi)$ is the polar coordinate given by $z = re^{i\phi}$.
For the proof of this identity, see \cite[Lemma 3.5]{SU} or \cite[Lemma 6.1.5]{LW}.

Recall the cylindrical coordinate $\phi : [-T_{k}^{\delta},T_{k}^{\delta}] \times S^{1} \rightarrow C_{k}$ given by \eqref{cyl coord}.
Then $\Phi \circ \phi : [-T_{k}^{\delta},T_{k}^{\delta}] \times S^{1} \rightarrow \overline{C_{k}}$ gives a coordinate change given by $\Phi \circ \phi (t,\theta) = x = \sqrt{t_{k}}e^{t+i\theta}$, so
\begin{equation*}
(F \circ \Phi \circ \phi)_{t} = \sqrt{t_{k}}e^{t}F_{r} = r F_{r}, \quad (F \circ \Phi \circ \phi)_{\theta} = F_{\phi}.
\end{equation*}
Then \eqref{Poho} becomes
\begin{equation*}
\int_{0}^{2\pi} \big\lvert (F \circ \Phi \circ \phi)_{\theta} \big\rvert^{2} d\theta = \int_{0}^{2\pi} \big\lvert (F \circ \Phi \circ \phi)_{t} \big\rvert^{2} d\theta.
\end{equation*}
Since $F \circ \Phi \circ \phi = f \circ \phi$, this proves the lemma.
\end{proof}

\begin{proof}
(Proof of Lemma \ref{gap neck})
By Proposition \ref{gap2} and Lemma \ref{reg node}.
\end{proof}

The assumption that $p$ is a regular node is non-avoidable.
Parker \cite{P} showed that there could be non-regular nodes by providing an example that the zero neck property fails.
In the language of Deligne-Mumford moduli space, his example can be seen as a family of torus with singular central fiber, which is a pinched torus.
Note that this node is not regular, because there is no forgetful map from this singular family to family with smaller marked points.
However, by modifying the example, one can get zero energy and zero (average) length of the neck, which satisfies the zero neck property.
(See \cite{Z} for details.)
This modification suggests the possibility that even over non-regular nodes, the zero neck property may hold, but not much research is done in this direction.

\section{Adding marked points to build new family}
\label{section6}

In this section we describe the procedure of adding marked points to build bigger family.
Let $f_{k} : C_{k} \rightarrow X$ be a sequence of harmonic maps with uniformly bounded energy where $C_{k}$ are smooth $(g,n)$ curves.
By Lemma \ref{loc conv}, $f_{k}$ converges to $f_{0} : C_{0} \rightarrow X$ off the singular set $S = Q \cup N$ where $Q$ is a set of smooth bubble points and $N$ is a set of nodal bubble points.
Denote $\overline{N} \subset N$ be a set of regular nodes.
Note that over smooth bubble point, energy concentration is at least $\varepsilon'_{0}$ and over regular nodal bubble point, energy concentration is at least $\varepsilon''_{0}$.
Together with $\varepsilon_{0}$ which comes from $\varepsilon$-regularity, we set
\begin{equation}
\bar{\varepsilon} = \frac{1}{2}\min{(\varepsilon_{0}, \varepsilon'_{0}, \varepsilon''_{0})}.
\end{equation}

Now we develop a framework of building bigger family.

\begin{lemma}\label{new family}
Fix a deformation $\pi : \mathcal{C} \rightarrow B$ and suppose $C_{k} = \pi^{-1}(b_{k})$, $C_{0} = \pi^{-1}(0)$ with $b_{k} \rightarrow 0$.
Suppose for each $k$ that
\begin{enumerate}
\item[{\sc {case 1}}]
For $p \in C_{0}$ a smooth unmarked point, there are two unmarked distinct points $q_{k},r_{k} \in C_{k}$ such that
\begin{equation*}
p = \lim_{k \rightarrow \infty}q_{k} = \lim_{k \rightarrow \infty}r_{k}.
\end{equation*}
\item[{\sc {case 2}}]
For $p \in C_{0}$ a nodal point, there is an unmarked point $r_{k} \in C_{k}$ such that
\begin{equation*}
p = \lim_{k \rightarrow \infty}r_{k}.
\end{equation*}
\end{enumerate}
Denote $C'_{k} = (C_{k},q_{k},r_{k})$ for Case 1, or $C'_{k} = (C_{k},r_{k})$ for Case 2.

Then there is a subsequence $n_{k}$, a nodal $(g,n+i)$ marked curve $C'_{0}$ where $i=1$ for Case 1 and $i=2$ for Case 2, a deformation $\pi' : \mathcal{C}' \rightarrow B'$ of $C'_{0}$ such that $C'_{n_k} \rightarrow C'_{0}$ in the family $\mathcal{C}'$, and forgetful map
\begin{equation} \label{forgetful map}
\begin{tikzcd}
\mathcal{C}' \arrow[r,"\Phi"] \arrow[d,"\pi'"] & \mathcal{C} \arrow[d,"\pi"]\\
B' \arrow[r] & B
\end{tikzcd}
\end{equation}
such that $\Phi(C'_{n_{k}}) = C_{n_{k}}$ and $\Phi(C'_{0}) = C_{0}$.
Here $\mathcal{C}'$ comes with two additional sections $\sigma$ and $\tau$ corresponding to $q_{k}$ and $r_{k}$ for Case 1, and with one additional section $\tau$ corresponding to $r_{k}$ for Case 2.

Furthermore, the restriction of $\Phi$ on $C'_{0}$ is the map collapsing a rational curve $E$, and is biholomorphic on $C'_{0} \setminus E$.
In Case 1, $E$ has two marked points $\sigma(0),\tau(0) \in E$ and one node, and in Case 2, $E$ has one marked point $\tau(0) \in E$ and two nodes.
\end{lemma}

\begin{proof}
Since the Deligne-Mumford moduli space $\overline{\mathcal{M}}_{g,n+i}$ is compact, there is a subsequence $n_{k}$ such that $[C'_{n_{k}}] \rightarrow [C'_{0}]$ for some $(g,n+i)$ curve $C'_{0}$.
Let
\begin{equation}\label{new family eq}
\begin{tikzcd}
\mathcal{C}' \arrow[d,"\pi'"]\\
B'
\end{tikzcd}
\end{equation}
be a Kuranishi family of $C'_{0} = (\pi')^{-1}(0)$.
For Case 1, this family comes with two sections $\sigma$ and $\tau$ corresponding to the last two marked points, and also with a forgetful map $\Phi$ as in \eqref{forgetful map}, which forgets $\sigma$ and $\tau$ and collapses unstable components.
For Case 2, this family comes with one section $\tau$ corresponding to the last marked point, and also with a forgetful map $\Phi$ as in \eqref{forgetful map}, which forgets $\tau$ and collapses unstable components.
Then we can choose $b'_{k} \rightarrow 0$ in $B'$ and for each $k$, an identification of $C'_{n_{k}}$ with a fiber $(\pi')^{-1}(b'_{k})$ of $\mathcal{C}'$ such that $q_{k} = \sigma(b'_{k})$ and $r_{k} = \tau(b'_{k})$ for Case 1, and $r_{k} = \tau(b'_{k})$ for Case 2.

In any case, the restriction of $\Phi$ on $C'_{0}$ is the map collapsing a rational curve $E$, and is biholomorphic on $C'_{0} \setminus E$.
Note that in Case 1, $E$ has two marked points $\sigma(0),\tau(0) \in E$ and one node, and in Case 2, $E$ has one marked point $\tau(0) \in E$ and two nodes.
(For details, see \cite{ACG} section 10.6 and 10.8)
\end{proof}

\begin{lemma}\label{new node regular}
Fix a family $\pi : \mathcal{C} \rightarrow B$ and suppose $C_{k} \rightarrow C_{0}$ in $\mathcal{C}$ and $p \in C_{0}$ is a regular node.
Also suppose that there is bigger family $\mathcal{C}'$ of $(g,n')$ curves with $n<n'$ and forgetful map $\Phi : \mathcal{C}' \rightarrow \mathcal{C}$.
Fix $C'_{0}$ be a fiber of $\mathcal{C}'$ such that $\Phi(C'_{0}) = C_{0}$.

Then any node $q \in C'_{0}$ with $\Phi(q) = p$ is regular.
\end{lemma}

\begin{proof}
Since $p$ is regular, there exists a family of $(g,l)$ curves $\overline{\pi} : \overline{\mathcal{C}} \rightarrow \overline{B}$ with $l<n$ and a forgetful map $\overline{\Phi} : \mathcal{C} \rightarrow \overline{\mathcal{C}}$ such that $\overline{p} := \Phi(p)$ is a regular point.
Then composition of forgetful maps $\overline{\Phi} \circ \Phi : \mathcal{C}' \rightarrow \overline{\mathcal{C}}$ is also a forgetful map from a family of $(g,n')$ curves to a family of $(g,l)$ curves.
Moreover, $\overline{\Phi} \circ \Phi (q) = \overline{p}$ is a regular point.
So $q$ is regular.
\end{proof}

Next, we specify where to put marked points near bubble points.
First consider $p \in Q$ be a smooth bubble point with energy concentration $m$.
Choose a neighborhood $U$ of $p$ and, given two distinct points $q,r \in U$, define (simplified) cross ratio $CR_{q,r} : U \rightarrow \mathbb{C}$ by
\begin{equation}
CR_{q,r}(x) = \frac{x-q}{r-q}.
\end{equation}
Note that $CR_{q,r}(q) = 0$ and $CR_{q,r}(r) = 1$.
Given $q_{k},r_{k}$, denote $R_{k} = CR_{q_{k},r_{k}}$.
Let $\mu_{k}$ be the energy density measure on $U$ and $\nu_{k} = (R_{k})_{*}\mu_{k}$ be the push forward measure on $R_{k}(B_{k}) \subset \mathbb{C}$.

\begin{lemma}\label{mark}
(Marking points near smooth bubble point)
Let $p \in Q$ be a smooth bubble point with energy concentration $m$.
Then after passing to a subsequence, there exist unmarked points $q_{k},r_{k} \in C_{k}$ both converging to $p$ such that
\begin{enumerate}
\item \label{cond1}
$C'_{k} = (C_{k},q_{k},r_{k}) \rightarrow C'_{0}$ in bigger family $\mathcal{C}'$ as in Lemma \ref{new family}.
\item \label{cond2}
Denote $f'_{k} = f_{k} \circ \Phi : C'_{k} \rightarrow X$.
Identify $E$ with $\mathbb{C}P^{1}$ by mapping $\sigma(0)$ to $[0:1]$, $\tau(0)$ to $[1:1]$, and the node $p'$ of $E$ to $[1:0]$.
Under the chart $[z:1] \mapsto z$,
\begin{align}
\lim_{k \rightarrow \infty}\int_{E \setminus D} \nu_{k} &= \bar{\varepsilon}, \label{mark-eq1}\\
\int_{D} z \,\nu_{k} &= 0 \label{mark-eq2}
\end{align}
where $D = \{z : \lvert z \rvert < 1\}$ and $\nu_{k} = e(\tilde{f}'_{k})$ are energy density measures on $E$.
\item \label{cond3}
On $E$, $\tilde{f}'_{k}$ converges to $f'_{0}$ in $C^{1}$ away from $\{p'\} \cup \{q_{j}\}_{j=1, \ldots, l'}$ where $q_{j} \in D \subset E$ with energy concentration $m_{j} \geq \varepsilon'_{0}$.
In addition, new node $p'$ is regular and
\begin{equation}
E(f'_{0}\vert_{E}) + \sum_{j=1, \ldots, l'}m_{j} = m.
\end{equation}
\end{enumerate}
\end{lemma}

Here the map $R_{k}$ acts as coordinate change from $U$ to $E$.
From the choice of $q_k,r_k$, the push-forward measure $\nu_{k}$ is such that
\begin{enumerate}
\item
essentially all of the mass $m$ is captured by $\nu_{k}$ over $E$,
\item
all but $\bar{\varepsilon}$ of that mass lies outside the unit disk $D$, and
\item
the center of mass of $\nu_{k}$ over $D$ is at the origin.
\end{enumerate}
The proof of this lemma is technical and will be in the appendix.

\bigskip

Now consider $p \in \overline{N}$ be a regular nodal bubble point with energy concentration $m$.
Nodal chart of $p$ can be written as $ B(p,\delta) \cap C_{k} = \{x,y \in \mathbb{C}^{2} : xy = t_{k}, \lvert x \rvert, \lvert y \rvert \leq \delta\}$.
Let $\pi_1$ be the projection to the first factor, given by $\pi_{1} (x,y) = x$, and denote $A_{k,\delta} := \pi_{1}(B(p,\delta) \cap C_{k}) = B_{\delta} \setminus B_{t_{k}/\delta} \subset \mathbb{C}$ where $B_{\delta}$ is a ball of radius $\delta$ centered at origin in the complex plane.
Define extended push forward energy density measure $\mu_{k}$ over $U := B_{\delta}$ by $\mu_{k} = (\pi_{1})_{*}e(f_{k})$ on $A_{k,\delta}$ and $\mu_{k} = 0$ on $B_{t_{k}/\delta}$.
Then $\mu_{k} \rightarrow \mu_{\infty} + (m + E_{\delta}) \delta_{p}$ where $\mu_{\infty} = e(f_{0})$ on $B_{\delta}$, $E_{\delta} = E(f_{0},\left(B(p,\delta) \cap C_0 \right)|_{\{x=0\}})$ and $\delta_{0}$ is Dirac-delta measure centered at $p$ and $m \geq 2\bar{\varepsilon}$.
By choosing $\delta$ small, we can make $E_{\delta}$ as small as we want.
Therefore, without loss of generality, we just denote $m$ instead of $m+E_{\delta}$ and consider
\begin{equation*}
\mu_{k} \rightarrow \mu_{\infty} + m \delta_{p}.
\end{equation*}

Given $r_{k}$, denote $R_{k} = CR_{p,r_{k}}$ and $\nu_{k} = (R_{k})_{*}\mu_{k}$ as above.

\begin{lemma}\label{mark2}
(Marking points near regular nodal bubble point)
Let $p \in \overline{N}$ be a regular nodal bubble point with energy concentration $m$.
Then after passing to a subsequence, there exist unmarked points $r_{k} \in C_{k}$ converging to $p$ such that
\begin{enumerate}
\item \label{cond1'}
$C'_{k} = (C_{k},r_{k}) \rightarrow C'_{0}$ in bigger family $\mathcal{C}'$ as in Lemma \ref{new family}.
\item \label{cond2'}
Denote $f'_{k} = f_{k} \circ \Phi : C'_{k} \rightarrow X$.
Identify $E$ with $\mathbb{C}P^{1}$ by mapping one node $p_{2}$ of $E$ to $[0:1]$, $\tau(0)$ to $[1:1]$, and the other node $p_{1}$ of $E$ to $[1:0]$.
Under the chart $[z:1] \mapsto z$,
\begin{equation} \label{mark2-eq}
\lim_{k \rightarrow \infty}\int_{E \setminus D} \nu_{k} = \bar{\varepsilon}
\end{equation}
where $D = \{z : \lvert z \rvert < 1\}$ and $\nu_{k} = e(\tilde{f}'_{k})$ are energy density measures on $E$.
\item \label{cond3'}
On $E$, $\tilde{f}'_{k}$ converges to $f'_{0}$ in $C^{1}$ away from $\{p_{1}\} \cup \{p_{2}\} \cup \{q_{j}\}_{j=1, \ldots, l'}$ where $q_{j} \in D \subset E$ with $q_{j} \neq p_{2}$ and with energy concentration $m_{j} \geq \varepsilon'_{0}$.
In addition, new nodes $p_{1},p_{2}$ are regular and
\begin{equation}
E(f'_{0}\vert_{E}) + \sum_{j=1, \ldots, l'}m_{j} + m_{0} = m
\end{equation}
where $m_{0}$ is the energy concentration at $p_{2} \in E$ and $m_{0} \geq \varepsilon''_{0}$.
\end{enumerate}
\end{lemma}
The proof of this lemma will be also in the appendix.

\section{Completion of the proof - Induction}
\label{section7}

In this section we prove Main Theorem \ref{main theorem} and its special cases, Corollaries \ref{main2}, \ref{fixed domain} and \ref{sphere}.
First we define the residual energy.

\begin{definition} \label{RE}
Suppose $f_{k} : C_{k} \rightarrow X$ be a sequence of maps with uniformly bounded energy that converges to $f_{0} : C_{0} \rightarrow X$ off the set $S = Q \cup N$, where $Q = \{p_{1}, \ldots, p_{l}\}$ is a set of smooth bubble points and $N$ is a set of nodal bubble points.
Let $\overline{N} = \{q_{1}, \ldots, q_{n}\}$ be a set of regular nodes that energy concentrates.

Define residual energy, denoted by $RE$, by
\begin{equation}
RE = \lim_{k \rightarrow \infty}E(f_{k}) - E(f_{0}) - l \bar{\varepsilon} - n \bar{\varepsilon}/2.
\end{equation}
If we denote $m_{i}$ be energy concentration at $p_{i}$ and $m'_{j}$ be energy concentration at $q_{j}$, then since $m_{i}, m'_{j} \geq 2\bar{\varepsilon}$, we have
\begin{equation*}
RE = \sum_{i=1}^{l}{(m_{i} - \bar{\varepsilon})} + \sum_{j=1}^{n} {(m'_{j} - \bar{\varepsilon}/2)} \geq l \bar{\varepsilon} + 3 n \bar{\varepsilon}/2.
\end{equation*}
\end{definition}

\begin{proof}
(Proof of Theorem \ref{main theorem})
If there is a subsequence $n_{k}$ and a finite set of points $P_{k}$ on $C_{k}$ such that $C'_{n_{k}} = (C_{n_{k}},P_{n_{k}}) \rightarrow C'_{0}$ for some $C'_{0}$ and corresponding residual energy $RE = 0$, then there is no energy concentration points except non-regular nodes and we are done.

Now suppose $RE > 0$ for any subsequence $n_{k}$ and any set of marking points $P_{k}$ such that $C'_{n_{k}}$ converges.
That means, energy concentration occurs at either $p \in Q$ or at $p \in \overline{N}$.

{\bf Case 1:}
There is energy concentration $m$ at $p \in Q$.

By Lemma \ref{mark}, after passing to a subsequence, we can add two marked points $q_{k},r_{k} \in C'_{k}$ such that $E(f'_{0}\vert_{E}) + \sum_{j=1}^{l'}m_{j} = m$, where $m_{j} \geq 2\bar{\varepsilon}$.

In the new family,
\begin{equation*}
RE' = \lim_{k \rightarrow \infty}E(f'_{k}) - E(f'_{0}) - (l-1+l') \bar{\varepsilon} - n \bar{\varepsilon}/2.
\end{equation*}
Note that $n$ does not change because the new node is a regular node with no energy concentration.
Then the difference of new residual energy from old one is
\begin{equation*}
RE' - RE = - (l'-1) \bar{\varepsilon} - E(f'_{0}\vert_{E}).
\end{equation*}

If $E(f'_{0}\vert_{E}) > 0$, then since $E(f'_{0}\vert_{E}) \geq 2\bar{\varepsilon}$, $RE' \leq RE - \bar{\varepsilon}$.
If $E(f'_{0}\vert_{E}) = 0$ and $l' \geq 2$, then $RE' \leq RE - \bar{\varepsilon}$.
Finally, if $E(f'_{0}\vert_{E}) = 0$ and $l' \leq 1$, we know $l'=1$ because of the energy identity $E(f'_{0}\vert_{E}) + \sum_{j=1}^{l'}m_{j} = m$.
Note that from \eqref{mark-eq2}, the location of the bubble on $E$ is $[0:1]$.
But then \eqref{mark-eq1} implies that energy of amount of $\bar{\varepsilon}$ on a subset of $E \setminus D$ can not be used for the bubble, hence $E(f'_{0}\vert_{E}) \geq \bar{\varepsilon}$.
This contradicts to the assumption $E(f'_{0}\vert_{E}) = 0$, so this case is impossible.

Hence, in any case, $RE' \leq RE - \bar{\varepsilon}$.

\smallskip

{\bf Case 2:}
There is energy concentration $m$ at a regular node $p \in \overline{N}$.

By Lemma \ref{mark2}, we can add one marked point $r_{k} \in C'_{n_{k}}$ and a subsequence $n'_{k}$ of $n_{k}$ such that $E(f'_{0}\vert_{E}) + \sum_{j=1}^{l'}m_{j} + m_{\infty} = m$, where $\nu_{0} = e(f'_{0})$, $m_{j}\geq 2\bar{\varepsilon}$, and $m_{\infty}$ is either zero or at least $2\bar{\varepsilon}$.

In the new family,
\begin{equation*}
RE' = \lim_{k \rightarrow \infty}E(f'_{k}) - E(f'_{0}) - (l+l') \bar{\varepsilon} - n' \bar{\varepsilon}/2
\end{equation*}
where $n'$ is the number of new regular nodes that energy concentrates.
Note that $n'=n$ if $m_{\infty} \neq 0$ or $n'=n-1$ if $m_{\infty} = 0$, so the difference of new residual energy from old one is
\begin{equation*}
RE' - RE \leq - l' \bar{\varepsilon} + \bar{\varepsilon}/2 - E(f'_{0}\vert_{E}).
\end{equation*}

If $E(f'_{0}\vert_{E}) > 0$, then since $E(f'_{0}\vert_{E}) \geq 2\bar{\varepsilon}$, $RE' \leq RE - \bar{\varepsilon}$.
If $E(f'_{0}\vert_{E}) = 0$ and $l' \geq 1$, then $RE' \leq RE - \bar{\varepsilon}/2$.
Finally, if $E(f'_{0}\vert_{E}) = 0$ and $l' = 0$, then $m_{\infty} =  m$ and all energy concentrates at $[0:1]$.
But from \eqref{mark2-eq}, energy of amount of $\bar{\varepsilon}$ on a subset of $E \setminus D$ can not concentrate at $[0:1]$, which contradicts to $E(f'_{0}\vert_{E}) = 0$.
So this case is impossible.

Hence, in any case, $RE' \leq RE - \bar{\varepsilon}/2$.

\smallskip

In conclusion, if $RE>0$, we mark either two points near a bubble point or one point near regular nodal point and make $RE' \leq RE-\bar{\varepsilon}/2$.
Since $RE$ is finite, this process should stop when energy concentrates only at non-regular nodes.
This proves the theorem.
\end{proof}

\begin{proof}
(Proof of Corollary \ref{main2})
The only thing we need to show is that all nodes of $C'_{0}$ are regular.
Pick $p'$ be a node in $C'_{0}$.
Since forgetful map $\Phi : \mathcal{C}' \rightarrow \mathcal{C}$ maps $C'_{0}$ to $C_{0}$, $\Phi(p') \in C_{0}$ is either regular point or nodal point.
If $\Phi(p')$ is regular point, $p'$ is regular node by definition.
If $\Phi(p')$ is nodal point, by assumption, it is regular nodal point.
So by Lemma \ref{new node regular}, $p'$ is regular node.
Hence we can apply Theorem \ref{main theorem} and conclusion follows from the fact that the singular set $S$ is empty.
\end{proof}

\begin{proof}
(Proof of Corollary \ref{fixed domain})
Let $C_{k} = \Sigma$, then $C_{0} = \Sigma$ which do not have any node.
The corollary then follows from Corollary \ref{main2}.
\end{proof}

\begin{proof}
(Proof of Corollary \ref{sphere})
Choose a limit $C_{0}$ such that $C_{k} \rightarrow C_{0}$ in a family $\mathcal{C}$ of $(0,n)$ curves with $n \geq 3$.
It is enough to show that all nodes of $C_{0}$ are regular.
Note that if $n=3$, Deligne-Mumford moduli space $\overline{\mathcal{M}}_{0,3}$ is trivial and there is no node in $C_{0}$.

Assume $n \geq 4$.
Pick $p \in C_{0}$ be a node.
Consider a forgetful map $\Phi : \mathcal{C} \rightarrow \overline{\mathcal{C}}$ that forgets $n-3$ marked points and collapse unstable components.
Here $\overline{\mathcal{C}}$ is a family of $(0,3)$ curves, which is again trivial.
So $\Phi(p)$ is regular point and hence $p$ is regular node.
This proves the corollary.
\end{proof}

\section{Appendix - Marking points lemmas}
\label{section8}

This section describes the proof of Lemmas \ref{mark} and \ref{mark2}.

\subsection{Local properties}

Assume the convergence of measures
\begin{equation}\label{meas conv2-re}
\mu_{k} \rightarrow \mu_{\infty} + m\delta_{p}
\end{equation}
where $\mu_{k},\mu_{\infty}$ are measures on $U \subset \mathbb{C}$, $\delta_{p}$ is Dirac-delta measure at the origin and $m \geq 2\bar{\varepsilon}$.
Throughout this section, $B(x,r) \subset U$ denotes a ball of radius $r$ in $U$ centered at $x \in U$ and $D \subset \mathbb{C}$ be a unit disk.

\begin{lemma}
There are $\delta_{k}$, $\varepsilon_{k}$, and a subsequence of $\mu_{k}$, still denoted by $\mu_{k}$, such that $\delta_{k}, \varepsilon_{k} \rightarrow 0$ and satisfying the following:
\begin{equation}
\int_{B(0,\delta_{0})} d\mu_{\infty} \leq \varepsilon_{0} = \frac{\bar{\varepsilon}}{4} \qquad \textrm{ and } \qquad \int_{B(0,\delta_{k})} d\mu_{\infty} \leq \varepsilon_{k} \label{subseq1}
\end{equation}
for all $k$.
Moreover, given any $k$, for all $m$ with $1 \leq m \leq 2k$,
\begin{equation}
\left\lvert \int_{B(0,\delta_{m})} d\mu_{k} - d\mu_{\infty} - m \delta_{p} \right\rvert < \varepsilon_{m}. \label{subseq2}
\end{equation}
\end{lemma}

\begin{proof}
Choose $\delta_{0} > 0$ and $\varepsilon_{0} = \bar{\varepsilon}/4$ such that the first equation in \eqref{subseq1} holds.

From \eqref{meas conv2-re}, given any $\varepsilon < \varepsilon_{0}$, $\delta < \delta_{0}$, there is a subsequence $\mu_{n_{k}}$ of $\mu_{k}$ such that
\begin{equation} \label{subseq2'}
\left\lvert \int_{B(0,\delta)} d\mu_{n_{k}} - d\mu_{\infty} - m \delta_{p} \right\rvert < \varepsilon
\end{equation}
for all $k$.

Pick $\varepsilon_{k} \rightarrow 0$ with $\varepsilon_{k} \leq \varepsilon_{k-1}/2$.
Pick $\delta_{k} \rightarrow 0$ with $\delta_{k} \leq \delta_{k-1}/2$ such that the second equation in \eqref{subseq1} holds.

For $(\varepsilon_{1},\delta_{1})$, there exist a subsequence $\mu^{1}_{k}$ of $\mu_{k}$ such that Equation \eqref{subseq2'} holds with $(\varepsilon_{1},\delta_{1})$.
For $(\varepsilon_{2},\delta_{2})$, there exist further subsequence $\mu^{2}_{k}$ of $\mu^{1}_{k}$ such that Equation \eqref{subseq2'} holds with $(\varepsilon_{1},\delta_{1})$ and $(\varepsilon_{2},\delta_{2})$.
Keep going and choose diagonal of above, say $\mu^{k}_{k}$.
Finally rename $\mu_{k} = \mu^{2k}_{2k}$.
Then, given any $k$, for all $1 \leq m \leq 2k$, Equation \eqref{subseq2} holds and the lemma is proved.
\end{proof}

Denote $B_{k} = B(0,\delta_{k})$.
By the choice of $B_{k}$, we have
\begin{equation} \label{mkBk}
\lim_{k \rightarrow \infty} \int_{B_{k}} d\mu_{k} = m.
\end{equation}
For simplicity, we fix $k$ and denote $\mu_{k}$ by simply $\mu$.
We first clarify which assumption we will use.
\begin{assume}\label{assume mu}
Assume $\mu$ is a smooth finite mass measure on a bounded set $U \in \mathbb{C}$.
By choosing $k$ large enough we may assume
\begin{enumerate}
\item \label{k cond1}
$2\varepsilon_{k} + 2\varepsilon_{2k} < \bar{\varepsilon}$,
\item \label{k cond2}
$3 \delta_{2k-1} < \delta_{k}$,
\item \label{k cond3}
$E := \mu (B_{k}) > m - 2 \varepsilon_{k} > \bar{\varepsilon}$ by \eqref{subseq1} and \eqref{subseq2},
\item \label{k cond4}
$(E - \bar{\varepsilon} - 2\varepsilon_{k} - 2\varepsilon_{2k})/2 > 8 (\varepsilon_{k} + \varepsilon_{2k})$.
\end{enumerate}
\end{assume}

\begin{definition}
Given $q \in B_{k}$ and $t \in (0,1)$, define $r_{t} = q + t/(1-t)$.
Also define cross ratio $R_{q,t}(x) : U \rightarrow \mathbb{C}$ by
\begin{equation} \label{CR def}
R_{q,t}(x) = \frac{(x-q)}{(r_{t}-q)} = \frac{1-t}{t}(x-q) = (t^{-1}-1) (x-q).
\end{equation}
\end{definition}

Note that for fixed $q$,
\begin{itemize}
\item
as $t \rightarrow 0$, $r_{t} \rightarrow q$ and $R_{q,t}(x) \rightarrow \infty$ for all $x \neq q$.
\item
as $t \rightarrow 1$, $r_{t} \rightarrow \infty$ and $R_{q,t}(x) \rightarrow 0$ for all $x$.
\end{itemize}

\begin{lemma} \label{out measure1}
Let $\mu$ be as in Assumption \ref{assume mu}.
Given $q \in B_{k}$, there exists a unique $t=t_{q} \in (0,1)$ such that
\begin{equation}
\int_{R_{q,t_{q}}(B_{k}) \setminus D} (R_{q,t_{q}})_{*} d\mu =  \bar{\varepsilon}.
\end{equation}
\end{lemma}

\begin{proof}
Define a continuous function $f(t) : (0,1) \rightarrow [0,\infty)$ by
\begin{equation*}
f(t) = \int_{R_{q,t_{q}}(B_{k}) \setminus D} (R_{q,t})_{*} d\mu = \int_{A_{t}} d\mu
\end{equation*}
where $A_{t} = B_{k} \setminus {R_{q,t}}^{-1}(D) = \{x \in B_{k} : \lvert R_{q,t}(x) \rvert > 1\}$.
Now for $x \neq q$,
\begin{equation*}
\frac{\partial}{\partial t} \lvert R_{q,t}(x) \rvert = -t^{-2} \left\lvert x-q \right\rvert < 0
\end{equation*}
so $\{A_{t}\}$ is a family of sets strictly descending on $t$ hence $f(t)$ is strictly decreasing.
Note that
\begin{align*}
\lim_{t \rightarrow 0}f(t) &= \lim_{t \rightarrow 0}\int_{A_{t}} d\mu = \int_{B_{k} \setminus \{q\}} d\mu = E > \bar{\varepsilon}\\
\lim_{t \rightarrow 1}f(t) &= \lim_{t \rightarrow 1}\int_{A_{t}} d\mu = \int_{\emptyset} d\mu = 0 < \bar{\varepsilon}
\end{align*}
hence there exists a unique $t_{q}$ such that $f(t_{q}) = \bar{\varepsilon}$.
\end{proof}

\begin{definition}
For simplicity, denote $R_{q} = R_{q,t_{q}}$.
\end{definition}

\begin{lemma}\label{t cont}
The assignment $q \mapsto t_{q}$ is continuous on $q \in B_{k}$.
\end{lemma}

\begin{proof}
Denote $f(t) = t/(1-t)$.
Then, $x \in R_{q,t}^{-1}(D)$ implies $\lvert x-q \rvert \leq f(t)$.

Fix $q$ and $\varepsilon>0$.
Since $f^{-1}$ is continuous, there is $\delta>0$ such that if $\lvert f(t_{q}) - f(t) \rvert \leq \delta$, then $\lvert t_{q} - t \rvert \leq \varepsilon$.
Fix $q' \in B_{k}$ such that $\lvert q - q' \rvert \leq \delta$.
Now it is enough to show that $\lvert t_{q} - t_{q'} \rvert \leq \varepsilon$.

Define $t_{\pm}$ be such that $f(t_{\pm}) = f(t_{q}) \pm \delta$.
It is easy to see that $t_{+} < t_{-}$.
We will show that
\begin{equation*}
R_{q',t_{+}}^{-1}(D) \subset R_{q,t_{q}}^{-1}(D) \subset R_{q',t_{-}}^{-1}(D).
\end{equation*}
To see this, for $x \in R_{q,t_{q}}^{-1}(D)$, $\lvert x-q \rvert \leq f(t_{q})$ implies
\begin{equation*}
\lvert x-q' \rvert \leq f(t_{q}) + \lvert q - q' \rvert \leq f(t_{q}) + \delta = f(t_{+}),
\end{equation*}
so $w \in R_{q',t_{+}}^{-1}(D)$.
The other direction is similar.

Hence, from the definition of $t_{q}$, we have $t_{+} \leq t_{q'} \leq t_{-}$.
Therefore $\lvert f(t_{q}) - f(t_{q'}) \rvert \leq \delta$, which implies $\lvert t_{q} - t_{q'} \rvert \leq \varepsilon$.
\end{proof}

\begin{definition}\label{F def}
Denote $\nu_{q} = (R_{q})_{*} \mu$.
Define $F : B_{k} \rightarrow \mathbb{C}$ by
\begin{equation}
F(q) = \int_{D} z d\nu_{q} (z).
\end{equation}
\end{definition}

\begin{prop}\label{conti}
F(q) in Definition \ref{F def} is continuous on $q \in B_{k}$.
\end{prop}

\begin{proof}
From Lemma \ref{t cont} and Equation \eqref{CR def}, it is obvious that $R_{q} = R_{q,t_{q}}$ is continuous on $q \in B_{k}$.
Hence push-forward measure $\nu_{q} = (R_{q})_{*}\mu$ is also continuous on $q$, and $F(q)$ is also continuous on $q$.
\end{proof}

\begin{prop}\label{zero}
Let $F(q)$ be in Definition \ref{F def}.
There exists $q_{k} \in B_{2k-1}$ such that $F(q_{k}) = 0$.
\end{prop}

This proposition is not trivial.
For example, consider $F(q) = (q+2\delta_{k})/3$.
Then $\lvert F(q) \rvert \geq \delta_{k}/3 > 0$ for all $q$, which is not the desired result.
To avoid this case, we need the following lemma.

\begin{lemma}\label{angle}
Let $F(q)$ be in Definition \ref{F def}.
For any given $q \in \partial B_{2k-1}$,
\begin{equation*}
Re \left( \frac{F(q)}{-q} \right) > 0.
\end{equation*}
\end{lemma}

\begin{proof}
\begin{equation*}
F(q) = \int_{D} z d\nu_{q}(z) = \int_{R_{q}^{-1}(D)} R_{q}(x) d\mu(x) = \frac{1-t_{q}}{t_{q}}\int_{R_{q}^{-1}(D)} (x-q) d\mu(x).
\end{equation*}
Denote $f(x) = x-q$.
If $x \in B_{2k}$, then for $u=x/q$,
\begin{equation*}
Re \left( \frac{f(x)}{f(0)} \right) = Re (1-u) > 0
\end{equation*}
because $\lvert u \rvert \leq 1/2$.

Define sets $A,B,C$ by
\begin{align*}
A &:= \{ x \in R_{q}^{-1}(D) : x \in B_{2k}\}\\
B &:= \left\{ x \in R_{q}^{-1}(D) : x \not\in A \textrm{ and } Re \left(\frac{f(x)}{f(0)}\right) \geq 0\right\}\\
C &:= \left\{ x \in R_{q}^{-1}(D) : Re \left(\frac{f(x)}{f(0)}\right) < 0\right\}.
\end{align*}

Note that $\mu(A) + \mu(B) + \mu(C) = \mu(R_{q}^{-1}(D)) = E - \bar{\varepsilon}$ and $\mu(B) + \mu(C) \leq \mu(B_{k} \setminus B_{2k}) < 2\varepsilon_{k} + 2\varepsilon_{2k}$, by Equation \eqref{subseq2} for $m=k$ and $m=2k$.
So we have $\mu(A) = E - \bar{\varepsilon} - \mu(B) - \mu(C) \geq E - \bar{\varepsilon} - 2 \varepsilon_{k} - 2 \varepsilon_{2k}$ and $\mu(C) \leq 2\varepsilon_{k} + 2\varepsilon_{2k}$.

Now,
\begin{align*}
Re \left(\frac{F(q)}{-q} \frac{t_{q}}{1-t_{q}}\right) &= \int_{A \cup B \cup C} Re \left(\frac{f(x)}{f(0)} \right) d\mu(x)\\
&\geq \int_{A} Re\left( \frac{f(x)}{f(0)} \right) d\mu(x) + \int_{C} Re\left( \frac{f(x)}{f(0)} \right) d\mu(x)\\
&\geq \int_{A} \frac{\delta_{2k-1} - \delta_{2k}}{\delta_{2k-1}} d\mu - \int_{C} \frac{\delta_{2k-1} + 3\delta_{2k-1}}{\delta_{2k-1}} d\mu \\
&\geq \frac{1}{2} \mu(A) - 4 \mu(C) \geq \frac{1}{2} (E - \bar{\varepsilon} - 2\varepsilon_{k} - 2 \varepsilon_{2k}) - 8 (\varepsilon_{k} + \varepsilon_{2k}) > 0
\end{align*}
by Assumption \eqref{k cond4}.

Hence we get $Re \left( -F(q)/q \right) > 0$.
\end{proof}

\begin{proof}
(Proof of Proposition \ref{zero})
Note that by Lemma \ref{angle}, $F(\partial B_{2k-1})$ is a closed curve with nonzero index.
So $F(B_{2k-1})$ contains $0$, which means that there exists $q \in B_{2k-1}$ such that $F(q)=0$.
\end{proof}

\smallskip

Now we go back to original sequence with subscript $k$.
For $q_{k}$ in Proposition \ref{zero}, denote $t_{k}=t_{q_{k}}$, $r_{k} = r_{t_{k}}$ and $R_{k} = R_{q_{k}}$.

\begin{lemma}\label{conv to C}
$r_{k} \in B_{k}$.
Also, $R_{k}(B_{k}) \rightarrow \mathbb{C}$ as $k \rightarrow \infty$.
\end{lemma}

\begin{proof}
First we claim that for any $q \in B_{k}$, $B_{2k} \not\subset R_{q}^{-1}(D)$.

Suppose not.
Then we have
\begin{equation*}
\bar{\varepsilon} = \int_{B_{k}\setminus R_{q}^{-1}(D)} d\mu_{k}  < \int_{B_{k}\setminus B_{2k}} d\mu_{k} \leq 2\varepsilon_{k} + 2\varepsilon_{2k} < \bar{\varepsilon}
\end{equation*}
which is a contradiction, so proves the claim.

So, there exists $x_{0} \in B_{2k}$ such that $x_{0} \not\in R_{q}^{-1}(D)$.
$\lvert R_{q}(x_{0}) \rvert > 1$ and $\lvert x_{0} \rvert \leq \delta_{2k}$ implies that
\begin{equation*}
\frac{t_{q}}{1-t_q} \leq \lvert q \rvert + \delta_{2k}.
\end{equation*}

Therefore,
\begin{equation*}
\lvert r_{k} \rvert \leq \lvert q_{k} \rvert + \frac{t_{k}}{1-t_{k}} \leq 2 \lvert q_{k} \rvert + \delta_{2k} \leq 3\delta_{2k-1} \leq \delta_{k}
\end{equation*}
which proves the first.

To show the second, it is enough to show that for any $R>0$, for all $k$ large enough, $R_{k}^{-1}(D_{R}) \subset B_{k}$ where $D_{R} \subset \mathbb{C}$ is a disk of radius $R$.
Fix $R>0$ and choose $x \in R_{k}^{-1}(D_{R})$.
Then we have
\begin{equation*}
\lvert x \rvert \leq \lvert q_{k} \rvert + R\frac{t_{k}}{1-t_{k}} \leq \delta_{2k-1} + R (\delta_{2k-1} + \delta_{2k}) \leq (1+2R)2^{-k+1}\delta_{k}.
\end{equation*}
Now choose $k$ large enough so that $(1+2R)2^{-k+1} \leq 1$.
\end{proof}

\subsection{Family properties}

Now we are ready to prove Lemmas \ref{mark} and \ref{mark2}.

\begin{proof}
(Proof of Lemma \ref{mark})
The energy density measures $\mu_{k}$ of $f_k$ on a neighborhood $U \subset \mathbb{C}$ of $p$ satisfies $\mu_{k} \rightarrow \mu_{\infty} + m\delta_{p}$ as measures and $m \geq 2\bar{\varepsilon}$, where $\mu_{\infty}$ is the energy density measure of $f_0$.
Choose $q_{k}$ as in Proposition \ref{zero} and $r_{k}$ as in Lemma \ref{out measure1}.
We abuse the notation $q_{k}$ and $r_{k}$ to refer points $(q_{k},b_{k}),(r_{k},b_{k}) \in C_{k}$ in regular chart, and mark them.
Since $\lim_{k}q_{k} = \lim_{k}r_{k} = p$, by Lemma \ref{new family}, we have new family of curves $\mathcal{C}'$ and forgetful map $\Phi : \mathcal{C}' \rightarrow \mathcal{C}$ such that $C'_{k} = (C_{k},q_{k},r_{k}) \rightarrow C'_{0}$ in $\mathcal{C}'$, which proves (1).

To see (2), we first describe coordinate expression in the family $\mathcal{C}'$ near $E \simeq \mathbb{C}P^{1}$ which agrees with coordinate of $\mathcal{C}$ near $p$ under forgetful map $\Phi$.
Description about coordinates can be found with more details in \cite{ACG} section 10.8.

Equip regular chart near $p$ by $(x,b) \in U \times B$ where $U \subset \mathbb{C}$ and $p = (0,0)$.
$\Phi(\sigma), \Phi(\tau)$ are sections in $\mathcal{C}$ which meet at $p$.
By adding $q$ in the base, we can see $\Phi(\sigma)$ as a marked section in this new family.
Locally, we consider the chart $(x,q,b) \in U \times U \times B \rightarrow (q,b)$ and view $\Phi(\sigma)$ and $\Phi(\tau)$ as functions on $U \times B$ with values on $U$ given by $\Phi(\sigma)(q,b) = q$ and $\Phi(\tau)(q,b) = r$ for some functions $q,r$ on $B$.

Note that for any fiber $C_{b}$ in $\mathcal{C}$, $\Phi^{-1}(C_{b})$ is a two parameter subfamily of the family $\mathcal{C}'$, consisting of $(C_{b},q,r)$ and its limit case $q=r$, which is a $1$ dimensional subset of nodal family whose fibers look like $C_{b} \cup \mathbb{C}P^{1}$, parametrized by $q$ which denotes the gluing position.
Translate coordinate from $x$ to $x' = x-q$ so that $\Phi(\sigma)$ is given by $\{x'=0\}$ and $\Phi(\tau)$ is given by $\{x' = t\}$ where $t=r-q$.
Choose a homogeneous coordinate $[\lambda:\mu] \in \mathbb{C}P^{1}$.
This gives local coordinate of $\mathcal{C}'$ near $E$ given by
\begin{equation} \label{new chart1}
((x',t),q,b,[\lambda:\mu]) \in (U \times U) \times U \times B \times \mathbb{C}P^{1} \rightarrow (t,q,b) \in U \times U \times B
\end{equation}
with equation
\begin{equation} \label{eq in C'-1}
x' \mu = t \lambda,
\end{equation}
and $\sigma$ and $\tau$ in $\mathcal{C}'$ can be written by equations
\begin{equation}
\lambda = 0 \quad \textrm{ and } \quad \lambda = \mu
\end{equation}
respectively.
Hence $\sigma$ and $\tau$ have coordinates $[0:1]$ and $[1:1]$ in $E$ respectively.
Furthermore, by choosing chart near the new node $p' = [1:0]$ in $E$ as $[\lambda : \mu] = [1:z'] \mapsto z'$ with $z' = \mu/\lambda$, Equation \eqref{eq in C'-1} can be written by
\begin{equation}
x' z' = t
\end{equation}
which is the same as nodal chart.

$(x-q,r-q,q,b,[\lambda:\mu])$ maps to $(x,b)$ under forgetful map $\Phi$ and projected to $[\lambda:\mu] \in E$ under local trivialization.
Consider the chart $[y:1] \mapsto y$ away from $p'$.
Note that $\Phi(C'_{k}) = C_{k}$ and $f'_{k} = f_{k} \circ \Phi : C'_{k} \rightarrow X$ is also a sequence of harmonic maps and
\begin{equation*}
\tilde{f}_{k}(x) = f_{k}(x,b_{k}) = f'_{k}(x-q_{k},r_{k}-q_{k},q_{k},b_{k},[\lambda:\mu]) = \tilde{f}'_{k}(z)
\end{equation*}
with $z = \lambda/\mu =(x-q_{k})/(r_{k}-q_{k}) = R_{k}(x)$, so we have $\nu_{k} = (R_{k})_{*}(e(\tilde{f}_{k})) = e(\tilde{f}'_{k})$.
Note that $\nu_{k}$ can extend to the whole $E$ by Lemma \ref{conv to C}.
Since the choice of $q_{k}$ and $r_{k}$ come from Lemmas \ref{out measure1} and \ref{zero} and cross ratio is conformally invariant, Equations \eqref{mark-eq1} and \eqref{mark-eq2} follows.
This proves (2).

Now consider (3).
By applying Lemma \ref{loc conv} again, there is a subsequence and a finite set of bubble points $\{q_{1}, \ldots, q_{l}\} \subset E \setminus \{p'\}$ such that after passing to a subsequence, $\nu_{k} \rightarrow e(f'_{0}) + \sum_{j}m_{j}\delta_{q_{j}}$ on $E \setminus \{p'\}$ with $m_{j} \geq \varepsilon'_{0}$.
Here $f'_{0} : C'_{0} \rightarrow X$ is a limit of $f'_{k}$.
By Equation \eqref{mark-eq1}, $q_{j} \in D$.
Denote $m_{\infty}$ the amount of energy concentration at $p'$, then we have
\begin{equation}
e(f'_{0})(E) + \sum_{j}m_{j} + m_{\infty} = m.
\end{equation}

For any compact set $K = \{[1:z'] : \lvert z' \rvert \geq \delta \} \subset \subset E \setminus \{p'\}$, define $B' = \{(x,b_{k}) \in C_{k} : x \in B_{k}\}$ and $K' = \{(x-q_{k},r_{k}-q_{k},q_{k},b_{k},[\lambda:\mu]) \in C'_{k} : [\lambda:\mu] \in K\}$.
Then
\begin{equation*}
m_{\infty} \leq m - e(f'_{0})(K) - \sum_{j}m_{j} = \lim_{k \rightarrow \infty} \left(\mu_{k}(B_{k}) - \nu_{k}(K) \right) = \lim_{k \rightarrow \infty} E(f'_{k},\Phi^{-1}(B') \setminus K').
\end{equation*}

We first show that for given $\delta>0$, $\Phi^{-1}(B') \setminus K' \subset B(p',\delta) \cap C'_{k}$ for all $k$ sufficiently large.
Pick $u = (x-q_{k},r_{k}-q_{k},q_{k},b_{k},[\lambda:\mu]) \in \Phi^{-1}(B') \setminus K'$.
Because $x,q_{k},r_{k} \in B_{k}$, $\lvert x-q_{k} \rvert, \lvert r_{k}-q_{k} \rvert, \lvert q_{k} \rvert, \lvert b_{k} \rvert \leq \delta$ for all $k$ sufficiently large.
Moreover, $u \notin K'$ means $[\lambda:\mu] \notin K$, which implies $\lvert \mu/\lambda \rvert = \lvert z' \rvert < \delta$.
So, $u \in B(p',\delta) \cap C'_{k}$ as desired.

Next, we will show that for any $\varepsilon>0$, there is $\delta>0$ such that $E(f'_{k}, B(p',\delta) \cap C'_{k}) < \varepsilon$ for all $k$ sufficiently large.

Fix $\varepsilon>0$ and assume $\delta<1$.
By Equation \eqref{mark-eq1}, $E(f'_{k},\Phi^{-1}(B') \setminus K') \leq \bar{\varepsilon} \leq \varepsilon''_{0}/2$.
Hence, for $\delta$ small enough, we have $E(f'_{k}, B(p',\delta) \cap C'_{k}) < \varepsilon''_{0}$.
By definition $p'$ is a regular node, so by Lemma \ref{gap neck}, there is $\delta>0$ such that $\lim_{k \rightarrow \infty}E(f'_{k}, B(p',\delta) \cap C'_{k}) < \varepsilon$.

So $m_{\infty} = 0$ and this proves the lemma.
\end{proof}

\begin{proof}
(Proof of Lemma \ref{mark2})
As description before Lemma \ref{mark2}, we have energy density measures $\mu_{k}$ on $U = B_{\delta} \subset \mathbb{C}$ satisfying $\mu_{k} \rightarrow \mu_{\infty} + m\delta_{p}$ as measures and $m \geq 2\bar{\varepsilon}$, where $\mu_{\infty}$ is the energy density measure of $f_0$ on $U$.
Choose $r_k$ as in Lemma \ref{out measure1} for $q=p$ which is the origin in $U$.
Again we abuse the notation $r_k$ to refer point $(r_{k},t_{k}/r_{k},\tilde{b}_{k}) \in C_{k}$ in nodal chart, and mark them.

Here we need to check $t_{k}/r_{k} \rightarrow 0$.
Suppose that there is $c>0$ such that $\lvert t_{k}/r_{k} \rvert \geq c$ for all $k$.
Since $\mu_{k} = 0$ on $B_{t_{k}/\delta}$ and $B_{c r_{k}/\delta} \subset B_{t_{k}/\delta}$, $\mu_{k} = 0$ on $B_{c r_{k}/\delta}$ for all $k$.
Recall that $R_{k}(x) =  CR_{p,r_{k}}(x)= x/r_{k}$.
Choose $\delta$ small enough such that $R_{k}^{-1}(D) = B_{r_{k}} \subset B_{c r_{k}/\delta}$ for all $k$.
Therefore $\mu_{k} = 0$ on $R_{k}^{-1}(D)$.
Now Equation \eqref{mark2-eq} can be rewritten as $\int_{B_{k} \setminus R_{k}^{-1}(D)} d\mu_{k} = \mu_{k}(B_{k}) = \bar{\varepsilon}$, which contradicts $\mu_{k}(B_{k}) \rightarrow m \geq 2\bar{\varepsilon}$.
Therefore $t_{k}/r_{k} \rightarrow 0$.

Since $\lim_{k}r_{k} = p$, by Lemma \ref{new family}, we have new family of curves $\mathcal{C}'$ and forgetful map $\Phi : \mathcal{C}' \rightarrow \mathcal{C}$ such that $C'_{k} = (C_{k},r_{k}) \rightarrow C'_{0}$ in $\mathcal{C}'$, which proves (1).

To see (2), we first describe coordinate expression in the family $\mathcal{C}'$ near $E \simeq \mathbb{C}P^{1}$ which agrees with coordinate of $\mathcal{C}$ near the node $p$ under forgetful map $\Phi$.
For details, see \cite{ACG} section 10.8.

Equip nodal chart near $p$ by $(x,y,\tilde{b}) \in U_{1} \times U_{2} \times \tilde{B}$ such that $xy=t$, where $U_{i} \subset \mathbb{C}$ and $p = (0,0,0)$.
The projection $\pi$ is locally given by $\pi(x,y,\tilde{b}) = (t,\tilde{b}) \in B$.
$\Phi(\tau)$ is a section in $\mathcal{C}$ which pass the node $p$.
Using the nodal chart, we can see $\Phi(\tau)$ as a vector-valued function on $B$ given by $\Phi(\tau) = (r,r')$ for some functions $r,r'$ on $B$ such that $r r' = t$.

Note that, for nodal fiber $C_{0}$ in $\mathcal{C}$ with node $p$, $\Phi^{-1}(C_{0})$ is a one parameter subfamily of the family $\mathcal{C}'$, consisting of $(C_{0},r)$ and its limit case where $r=p$, which looks like $C_{0} \cup \mathbb{C}P^{1}$.
Choose a homogeneous coordinate $[\lambda:\mu] \in \mathbb{C}P^{1}$.
This gives local coordinate of $\mathcal{C}'$ near $E$ given by
\begin{equation}\label{new chart2}
((x,y,r,r'),\tilde{b},[\lambda:\mu]) \in (U_{1} \times U_{2} \times U_{1} \times U_{2}) \times \tilde{B} \times \mathbb{C}P^{1} \rightarrow (r,r',\tilde{b}) \in U_{1} \times U_{2} \times \tilde{B}
\end{equation}
with equations
\begin{equation} \label{eq in C'-2}
\lambda r = \mu x \quad \textrm{ and } \quad \lambda y = \mu r' 
\end{equation}
and $\tau$ in $\mathcal{C}'$ can be written by equation
\begin{equation}
\lambda = \mu.
\end{equation}
So $\tau$ has coordinate $[1:1]$ in $E$.

Now $E$ has two nodes, $p_{1} = [1:0]$ and $p_{2} = [0:1]$.
Near $p_{1}$, choose chart of $E$ by $[1:z'] \mapsto z'$ with $z' = \mu/\lambda$.
Then the first equation in \eqref{eq in C'-2} can be written by
\begin{equation}
x z' = r
\end{equation}
which is nodal chart near $p_{1} = [1:0]$.

On the other hand, near $p_{2}$, choose chart of $E$ by $[z:1] \mapsto z$ with $z = \lambda/\mu$.
Then the second equation in \eqref{eq in C'-2} can be written by
\begin{equation}
y z = r'
\end{equation}
which is again nodal chart near $p_{2} = [0:1]$.

$(x,y,r,r',\tilde{b},[\lambda:\mu])$ maps to $(x,y,\tilde{b})$ under forgetful map $\Phi$ and projected to $[\lambda:\mu] \in E$ under local trivialization.
Consider the chart $[z:1] \mapsto z$ away from both $p_{1},p_{2}$.
Note that $\Phi(C'_{k}) = C_{k}$ and $f'_{k} = f_{k} \circ \Phi : C'_{k} \rightarrow X$ is also a sequence of harmonic maps and
\begin{equation*}
(\pi_{1})^{*}f_{k}(x) = f_{k}(x,y,\tilde{b}_{k}) = f'_{k}(x,y,r_{k},t_{k}/r_{k},\tilde{b}_{k},[\lambda:\mu]) = \tilde{f}'_{k}(z)
\end{equation*}
with $z = \lambda/\mu = x/r_{k} = R_{k}(x)$, so we have $\nu_{k} = (R_{k})_{*}(\pi_{1})_{*}(e(f_{k})) = e(\tilde{f}'_{k})$.
Note that $\nu_{k}$ can extend to the whole $E$ by Lemma \ref{conv to C}.
Since the choice of $r_{k}$ come from Lemma \ref{out measure1} and cross ratio is conformally invariant, Equation \eqref{mark2-eq} follows.
This proves (2).

Now consider (3).
By applying Lemma \ref{loc conv} again, there is a subsequence and a finite set of bubble points $\{q_{1}, \ldots, q_{l}\} \subset E \setminus \{p_{1}, p_{2}\}$ such that after passing to a subsequence, $\nu_{k} \rightarrow e(f'_{0}) + \sum_{j}m_{j}\delta_{q_{j}}$ on $E \setminus \{p_{1},p_{2}\}$ with $m_{j} \geq \varepsilon'_{0}$.
Here $f'_{0} : C'_{0} \rightarrow X$ is a limit of $f'_{k}$.
By Equation \eqref{mark2-eq}, $q_{j} \in D$ and $q_{j} \neq p_{2}$.
Denote $m_{0}$ and $m_{\infty}$ the amount of energy concentration at $p_{2}$ and at $p_{1}$ respectively.
Then we have
\begin{equation}
e(f'_{0})(E) + \sum_{j}m_{j} + m_{0} + m_{\infty} = m.
\end{equation}
Since $p$ is regular node, $p_{1}$ and $p_{2}$ are also regular by Lemma \ref{new node regular}.
Hence $m_{0},m_{\infty}$ are either zero or at least $\varepsilon''_{0}$ by Lemma \ref{gap neck}.

For any compact set $K \subset \subset E \setminus \{p_{1},p_{2}\}$, define $B' = \{(x,y,\tilde{b}_{k}) \in B(p,\delta) \cap C_{k} : xy=t_{k}, x \in B_{k}\}$ and $K'  = \{(x,y,r_{k},t_{k}/r_{k},\tilde{b}_{k},[\lambda:\mu]) \in C'_{k} : [\lambda:\mu] \in K\}$.
Then
\begin{equation*}
m_{0} + m_{\infty} \leq m - e(f'_{0})(K) - \sum_{j}m_{j} \leq \lim_{k \rightarrow \infty} \left(\mu_{k}(B_{k}) - \nu_{k}(K) \right) = \lim_{k \rightarrow \infty} E(f'_{k},\Phi^{-1}(B') \setminus K').
\end{equation*}

We first show that, for any $\delta' < \delta$,
\begin{align*}
K_{1} &:= \{q \in \Phi^{-1}(B') \setminus K' : \lvert \mu/\lambda \rvert < \delta'\} \subset B(p_{1},\delta') \cap C'_{k},\\
K_{2} &:= \{q \in \Phi^{-1}(B') \setminus K' : \lvert \lambda/\mu \rvert < \delta'\} \subset  B(p_{2},\delta) \cap C'_{k}
\end{align*}
for all $k$ sufficiently large.
Let $u = (x,y,r_{k},t_{k}/r_{k},\tilde{b}_{k},[\lambda:\mu]) \in K_{1}$.
Note that because $x,r_{k} \in B_{k}$, $\lvert x \rvert, \lvert r_{k} \rvert, \lvert \tilde{b}_{k} \rvert \leq \delta'$ for all $k$ sufficiently large.
Moreover, $\lvert \mu/\lambda \rvert = \lvert z' \rvert <\delta'$ for all $k$ sufficiently large.
Therefore $u \in  B(p_{1},\delta') \cap C'_{k}$ as desired.

On the other hand, let $u = (x,y,r_{k},t_{k}/r_{k},\tilde{b}_{k},[\lambda:\mu]) \in K_{2}$.
As above, $\lvert t_{k}/r_{k} \rvert, \lvert \tilde{b}_{k} \rvert \leq \delta'$ for all $k$ sufficiently large.
We also have $\lvert t_{k}/x \rvert = \lvert y \rvert \leq \delta$.
Moreover, $\lvert \lambda/\mu \rvert = \lvert z \rvert < \delta'$ for all $k$ sufficiently large.
Therefore $u \in B(p_{2},\delta) \cap C'_{k}$ as desired.

Next, we will show that for any $\varepsilon>0$, there is $\delta'$ such that $E(f'_{k}, B(p_{1},\delta') \cap C'_{k}) \leq \varepsilon$ for all $k$ sufficiently large.

Fix $\varepsilon>0$ and assume $\delta' < 1$.
By Equation \eqref{mark2-eq}, $E(f'_{k},K_{1}) \leq \bar{\varepsilon} \leq \varepsilon''_{0}/2$.
Hence, for $\delta'$ small enough, we have $E(f'_{k}, B(p_{1},\delta') \cap C'_{k}) < \varepsilon''_{0}$.
Therefore, by Lemma \ref{gap neck}, there is $\delta'>0$ such that $\lim_{k \rightarrow \infty}E(f'_{k}, B(p_{1},\delta') \cap C'_{k}) < \varepsilon$.

This shows $m_{\infty} \leq \lim_{k \rightarrow \infty}E(f'_{k},K_{1}) \leq \lim_{k \rightarrow \infty}E(f'_{k}, B(p_{1},\delta') \cap C'_{k}) < \varepsilon$ for any $\varepsilon>0$.
So $m_{\infty} = 0$ and this proves the lemma.
\end{proof}


\bibliographystyle{amsplain}

\end{document}